\documentclass[12pt]{amsart}
\usepackage{amssymb,amsmath,amsfonts,latexsym}
\usepackage{bm}
\usepackage{array,graphics,color}
\newcommand{\newsection}[1]{\setcounter{equation}{0} \section{#1}}
\numberwithin{equation}{section}

\newtheorem{propn}{Proposition}[section]
\newtheorem{thm}[propn]{Theorem}
\newtheorem{lemma}[propn]{Lemma}

\newtheorem*{thm*}{Theorem}
\theoremstyle{definition}
\newtheorem{defn}[propn]{Definition}

\newtheorem{rem}[propn]{Remark}

\newcommand{\Hil}{\mathcal{H}}

\newcommand{\Z}{\mathbb{Z}_+}
\newcommand{\sz}{\mathbb{S}}

\newcommand{\T}{\mathcal{T}}
\newcommand{\p}{\mathcal{P}}

\newcommand{\h}{\hat{T}}

\newcommand{\Comp}{\mathbb{C}}

 \newcommand{\D}{\mathbb{D}}

 \newcommand{\Dt}{D_{T}}

 \newcommand{\dt}{\cld_{T}}

 \DeclareMathOperator{\Ran}{Ran}

\newcommand{\vp}{\varphi}

\newcommand{\clb}{\mathcal{B}}

\newcommand{\cld}{\mathcal{D}}
\newcommand{\cle}{\mathcal{E}}
\newcommand{\clf}{\mathcal{F}}

\newcommand{\clh}{\mathcal{H}}
\newcommand{\clk}{\mathcal{K}}
\newcommand{\cll}{\mathcal{L}}

\newcommand{\clp}{\mathcal{P}}
\newcommand{\clq}{\mathcal{Q}}

\newcommand{\cls}{\mathcal{S}}

\newcommand{\SA}{\mathcal{S}\mathcal{A}}
\newcommand{\z}{\bm{z}}
\newcommand{\w}{\bm{w}}
\newcommand{\raro}{\rightarrow}
\newcommand{\NI}{\noindent}

\begin{document}

\title[{Finite rank
commuting contractions}]{Isometric dilations and von Neumann
inequality for finite rank commuting contractions}

\author[Barik]{Sibaprasad Barik}
\address{Department of Mathematics, Indian Institute of Technology Bombay, Powai, Mumbai, 400076, India}
\email{sibaprasadbarik00@gmail.com}

\author[Das] {B. Krishna Das}
\address{Department of Mathematics, Indian Institute of Technology Bombay, Powai, Mumbai, 400076, India}
\email{dasb@math.iitb.ac.in, bata436@gmail.com}

\author[Sarkar]{Jaydeb Sarkar}
\address{Indian Statistical Institute, Statistics and Mathematics Unit, 8th Mile, Mysore Road, Bangalore, 560059, India}
\email{jay@isibang.ac.in, jaydeb@gmail.com}


%\today

\begin{abstract}
Motivated by Ball, Li, Timotin and Trent's Schur-Agler class version of commutant lifting
theorem, we introduce a class, denoted by $\mathcal{P}_n(\mathcal{H})$, of $n$-tuples of commuting contractions on a Hilbert space $\mathcal{H}$. We always assume that $n \geq 3$. The importance of this class of $n$-tuples stems from the fact that the von Neumann inequality or the existence of isometric dilation does not hold in general for $n$-tuples, $n \geq 3$, of commuting contractions on Hilbert spaces (even in the level of finite dimensional Hilbert spaces). Under some rank-finiteness assumptions, we prove that tuples in $\mathcal{P}_n(\mathcal{H})$ always admit explicit isometric dilations and satisfy a refined von Neumann inequality in terms of algebraic varieties in the closure of the unit polydisc in $\mathbb{C}^n$.
\end{abstract}

\maketitle

\newsection{Introduction}\label{sec-intro}

In this paper we deal with the problem of isometric dilations and (refined) von-Neumann inequality for $n$-tuples ($n \geq 3$) of commuting contractions. For notational convenience, we denote by $\T^n(\clh)$ the set of all ordered $n$-tuples of commuting contractions on a Hilbert space $\clh$:
\[
\T^n(\clh) = \{(T_1, \ldots, T_n): T_i \in \clb(\clh),  \|T_i\| \leq
1, T_i T_j = T_j T_i, 1 \leq i, j \leq n\},
\]
where $\clb(\clh)$ denotes the set of all bounded linear operators
on $\clh$. Let $V = (V_1, \ldots, V_n) \in \T^n(\clk)$ be a tuple of commuting isometries (that is, $V_i^* V_i = I_{\clk}$ for all $i=1, \ldots, n$) on a Hilbert space $\clk$, and let $\clk \supseteq \clh$ (or, $\clh$ is isometrically embedded in $\clk$). We say that $V$ is an \textit{isometric dilation} of $T \in \T^n(\clh)$ if
\[
T^{\bm{k}} = P_{\clh} V^{\bm{k}}|_{\clh} \quad \quad (\bm{k} \in \mathbb{Z}_+^n),
\]
where $P_{\clh}$ denotes the orthogonal
projection of $\clk$ onto $\clh$ and
\[
\mathbb{Z}_+^n = \{\bm{k} = (k_1, \ldots, k_n) : k_i \in
\mathbb{Z}_+, i = 1, \ldots, n\},
\]
and for each multi-index $\bm{k} \in \mathbb{Z}_+^n$ and commuting
tuple $A = (A_1, \ldots, A_n)$ on a Hilbert space $\cll$ we denote $A^{\bm{k}} = A_1^{k_1} \cdots A_n^{k_n}$. It is well known that the existence of an isometric dilation of $T \in \T^n(\clh)$ guarantees \cite{NF} the \textit{von-Neumann inequality} for $T$:
\[
\|p(T_1, \ldots, T_n)\|_{\clb(\clh)} \leq \mbox{sup} \{|p(\z)|: \bm{z} \in \overline{\D}^n\},
\]
for all $p \in \mathbb{C}[z_1, \ldots, z_n]$. Now one knows on one hand the existence of isometric dilations of $n$-tuples in $\T^n(\clh)$, $n = 1, 2$, is guaranteed by the celebrated dilation theory of Sz.-Nagy and Foias and Ando. On the other hand, in general, neither the existence of isometric dilation
nor the von Neumann inequality holds for tuples in $\T^n(\clh)$, $n > 2$ (even in the case of $\mbox{dim~} \clh < \infty$). For instance, see the counterexamples by Varopoulous \cite{V}, Crabb and Davie \cite{CD1} and Parrott \cite{Par}.

An intriguing question therefore is to identify those $n$-tuples in $\T^n(\clh)$, $n \geq 3$, which admit isometric dilation (and also satisfy the von Neumann inequality over a variety in $\overline{\D}^n$ or $\D^n$ in the sense of \cite{AM1}, \cite{DS}, \cite{DSS} and \cite{BDHS}). This has turned into one of the most challenging questions in multivariable operator theory and functions of several complex variables. However, the research in this direction seems unexplored except the work of Grinshpan,
Kaliuzhnyi-Verbovetskyi, Vinnikov and Woerdeman \cite{VV} and the
recent paper \cite{BDHS}. Also see Choi and Davidson \cite{CD}, Drury \cite{Dr}, Holbrook \cite{Ho1, Ho2}, Knese \cite{K1} and Kosi\'{n}ski \cite{LK} for relevant examples and results.

The complexity of the above problem is further compounded by a number of related (and equally complex) function theoretic problems in several complex variables like commutant lifting theorem, (Toeplitz) corona theorem, Nevanlinna-Pick interpolation theorem, Caratheodory-Fejer theorem and invariant subspace problem. Here we are particularly interested in the commutant lifting theorem, which is also equivalent to the Ando dilation theorem (cf. \cite{FF}). The commutant lifting theorem in the setting of scalar-valued Hardy space is due to Sarason \cite{Sarason}. We state here the Sz.-Nagy and Foias version \cite{NF} of commutant lifting theorem in the setting of vector-valued Hardy space: Let $\cle$ and $\cle_*$ be Hilbert spaces, $\clq \subseteq H^2_{\cle}(\D)$ and $\clq_* \subseteq H^2_{\cle_*}(\D)$ be closed subspaces and let $X \in \clb(\clq, \clq_*)$. Suppose that $\clq$ and $\clq_*$ are shift co-invariant subspaces (invariant under the adjoint of the multiplication operator $M_z$) of $H^2_{\cle}(\D)$ and $H^2_{\cle_*}(\D)$, respectively, and
\[
X (P_{\clq} M_{z}|_{\clq}) = (P_{\clq_*} M_{z}|_{\clq_*}) X.
\]
Then there exists a bounded holomorphic function $\vp \in H^\infty_{\clb(\cle, \cle_*)}(\D)$ such that $\|X\| = \|\vp\|_{\infty}$ (the uniform norm of $\vp$ over $\D$) and
\[
X^* = M_{\vp}^*|_{\clq_*}.
\]
Here and in what follows, $H^\infty_{\clb(\cle, \cle_*)}(\D^n)$ denotes the set of all bounded $\clb(\cle, \cle_*)$-valued analytic functions on $\D^n$. In connection with the above it is now natural to ask whether the commutant lifting theorem can be extended to the case of vector-valued Hardy space over the polydisc $\D^n$ in $\mathbb{C}^n$. This has been addressed by Ball, Li, Timotin and Trent \cite{BLTT} for a special class, known as Schur-Agler class, of multipliers. To be more specific, let $\cle$ and $\cle_*$ be Hilbert spaces. The Schur-Agler class $\SA_n(\cle, \cle_*)$ consists of $\clb(\cle, \cle_*)$-valued analytic functions $\Phi$ on $\D^n$ such that $\Phi$ satisfies the $n$-variables von Neumann inequality, that is
\[
{\mathcal{S}\mathcal{A}}_n(\cle, \cle_*) = \{ \Phi \in H^\infty_{\clb(\cle,
\cle_*)}(\D^n) : \|\Phi(T)\| \leq 1, T \in \T^n_1 (\clh)
~\text{and}~\clh \text{~a Hilbert space}\},
\]
where $\T^n_1(\clh) = \{T \in \T^n(\clh) : \|T_i \| < 1, i = 1,
\ldots, n\}$. Here, the functional calculus $\Phi(T)$ is given by
\[
\Phi(T) = SOT-\sum_{\bm{k} \in \mathbb{Z}_+^n} \Phi_{\bm{k}} \otimes T^{\bm{k}},
\]
where $\Phi(\z) = \sum_{\bm{k} \in \mathbb{Z}_+^n} \Phi_{\bm{k}} z^{\bm{k}}$ is the Taylor expansion for $\Phi$ centered at the origin in $\mathbb{C}^n$ with $\Phi_{\bm{k}} \in \clb(\cle, \cle_*)$ and $\z^{\bm{k}} = z_1^{k_1} \cdots z_n^{k_n}$ for all $\bm{k} \in \mathbb{Z}_+^n$. The elements of $\SA_n(\cle, \cle_*)$ are called Schur-Agler class functions. It is worth noting that the Schur-Agler class is a proper subset of bounded holomorphic functions in three or more than three variables \cite{CD, Ho1, Ho2, K1, LK, Par, V}:
\[
\SA_n(\cle, \cle_*) \subsetneqq H^\infty_{\clb(\cle, \cle_*)}(\D^n) \quad \quad (n > 2).
\]
We need one more piece of notation. Given $A \in \clb(\clh)$, the \textit{conjugate map} \cite{BLTT} is the completely positive map $C_A : \clb(\clh) \raro \clb(\clh)$ defined by
\begin{equation}\label{Cmap}
C_A(X) = A X A^* \quad \quad (X \in \clb(\clh)).
\end{equation}
It is easy to see that if $A_1 A_2 = A_2 A_1$ for some $A_1, A_2 \in \clb(\clh)$, then $C_{A_1} C_{A_2} = C_{A_2} C_{A_1}$.

\noindent We are now ready to state the Ball, Li, Timotin and Trent's Schur-Agler class version of commutant lifting theorem (Theorem 2.4 in \cite{BLTT}): Let $\clq \subseteq H^2_{\cle}(\D^n)$ and $\clq_* \subseteq H^2_{\cle_*}(\D^n)$ be closed subspaces and let $X \in \clb(\clq, \clq_*)$. Suppose that $M_{z_i}^* \clq \subseteq \clq$ and $M_{z_i}^* \clq_* \subseteq \clq_*$ and
\[
(P_{\clq_*} M_{z_i}|_{\clq_*}) X = X (P_{\clq} M_{z_i}|_{\clq}),
\]
for all $i = 1, \ldots, n$. Then there exists $\Phi \in \SA_n(\cle, \cle_*)$ such that $\|X\| = \|\Phi\|_{\infty}$ and
\[
X^*= M_{\Phi}^*|_{\clq_*},
\]
if and only if there exist positive operators $G_i \in \clb(\clq_*)$, $i = 1, \ldots, n$, such that
\[
I - X X^* = G_1 + \dots+G_{n},
\]
and
\[
\Big(\prod_{\stackrel{j=1}{j\ne i}}^{n}(I_{\clb(\clq_*)} - C_{P_{\clq_*} M_{z_i}|_{\clq_*}})\Big)(G_i) \ge 0,
\]
for all $i = 1,\ldots,n$.

Using the above Ball, Li, Timotin and Trent commutant lifting theorem as an inspiration, we introduce a class of operators in $\T^n(\clh)$. But before we do that, we introduce some notation and definitions.

\textsf{From now on we will assume that $n \geq 3$.} For $T \in \T^n(\clh)$ and $1 \leq i \leq n$, define
\[
\hat{T}_{i} = (T_1,\dots,T_{i-1}, T_{i+1},\dots,T_{n})
\in \T^{(n-1)}(\clh),
\]
the $(n-1)$-tuple obtained from $T$ by deleting $T_i$. Now we define the set
\[
\mathbb{S}_n(\clh) = \{ T \in \T^n(\clh) : \mathbb{S}_n^{-1}(T, T^*)
\geq 0 \; \mbox{and~} T_i \mbox{~is pure for all~} i = 1, \ldots, n\},
\]
where
\[
\mathbb{S}_n^{-1}(T, T^*) = \sum_{\bm{k} \in \{0,1\}^n}
(-1)^{|\bm{k}|} T^{\bm{k}} T^{*\bm{k}}.
\]
The elements of $\mathbb{S}_n(\clh)$ are called \textit{Szeg\"{o}
$n$-tuples}. Recall that a contraction $X \in \T^1(\clh)$ is said to be \textit{pure} if $\|X^{*m} h\| \raro 0$ as $m\raro\infty$ for all $h \in \clh$. Now we are ready to define the central object of this paper.

\begin{defn}\label{defn-PT}
An $n$-tuple $T \in \T^n(\clh)$ is said to be in $\p_n(\clh)$ if $\h_n \in
\mathbb{S}_{n-1}(\clh)$ and there exist positive operators $G_1, \ldots, G_{n-1}$ (depending on $T$) in $\clb(\clh)$ such that
\[
I-T_nT_n^* = G_1 + \dots+G_{n-1},
\]
and
\[
S_T(G_i) : = \Big(\prod_{\stackrel{j=1}{j\ne i}}^{n-1}(I_{\mathcal{B}(\Hil)}-
C_{T_j})\Big)(G_i) \geq 0,
\]
for all $i = 1,\ldots,n-1$.
\end{defn}

\begin{defn}\label{defn-Finite PT}
Let $T \in \p_n(\clh)$. We say that $T$ is a \textit{finite rank tuple} if
there exist positive operators $G_1, \ldots, G_{n-1}$ associated to $T$
as in the above definition such that
\[
\mbox{rank} \Big( \mathbb{S}_{n-1}^{-1}(\hat{T}_n, \hat{T}^*_n) \Big) < \infty, \quad \mbox{and} \quad \mbox{rank} \Big(S_T(G_i) \Big) < \infty,
\]
for all $i = 1, \ldots, n-1$.
\end{defn}

\textit{The main results of this paper says that}: If $T \in \p_n(\clh)$ is a finite rank tuple, then

(i) $T$ admits an explicit isometric dilation (see
Theorem \ref{dilation2}), and

(ii) there exists an algebraic variety $V$ in $\overline{\D}^n$ such that
\[
\|p(T)\|_{\clb(\clh)} \leq \sup_{\z \in V} |p(\z)|,
\]
for all $p \in \mathbb{C}[z_1, \ldots, z_n]$ (see  Theorem \ref{vN2}).

In fact, in Theorem \ref{key theorem} we first reprove the Ball, Li, Timotin and Trent commutant lifting theorem. Here, however, with a slightly more elaborated idea we prove an explicit version of the commutant lifting theorem. This method then yields an explicit construction of isometric dilations of finite rank tuples in $\T^n(\clh)$. This is the content of Section \ref{section-dilation}. Then in Section \ref{section-vn ineq}, as application of the explicit isometric dilations, we prove a refined version of von Neumann inequality, in terms of algebraic varieties in $\bar{\D}^n$ (or in $\D^n$), of finite rank tuples in $\p_n(\clh)$.

In Section \ref{section-examples}, we present some elementary examples. In Section \ref{section-preparatory}, we set up some notation and terminology and prove some basic results.

It is worth noting, in this context, that the class of commuting contractions in $\p_n(\clh)$ is larger than the one considered in \cite{BDHS} (see Remark \ref{remark-compare}).

\newsection{Examples}\label{section-examples}

Before we move on to the technical part, we present an elementary but non-trivial example of tuple in $\p_n(\clh)$.

Let $(T_1,T_2) \in \mathbb{S}_2(\clh)$ and let  $j, k \geq 1$. Suppose $T_3 = T_1^j T_2^k$. Then $T = (T_1, T_2, T_3) \in
\p_3(\clh)$. Indeed, if we set
\[
G_1 = I-T_1^j T_1^{*j},
\]
and
\[
G_2 = T_1^j (I - T_2^k T_2^{*k}) T_1^{*j},
\]
then clearly
\[
I - T_3 T_3^* = G_1 + G_2.
\]
On the other hand, $\mathbb{S}_2^{-1}((T_1,T_2), (T_1^*,T_2^*)) \geq 0$ implies that
\[
T_2(I-T_1T_1^*)T_2^*\le I-T_1T_1^*,
\]
from which it follows that
\[
\begin{split}
T_2 G_1T_2^* & = T_2 (\sum _{i=0}^{j-1}T_1^i(I-T_1T_1^*)T_1^{*i})
T_2^*
\\
&=\sum_{i=0}^{j-1}T_1^{i}T_2(I-T_1T_1^*)T_2^*T_1^{* i}
\\
&\le \sum_{i=0}^{j-1}T_1^{i}(I-T_1T_1^*)T_1^{* i}
\\
&= I-T_1^jT_1^{*j},
\end{split}
 \]
that is
\[
G_1 - T_2 G_1 T_2^* \geq 0.
\]
Similarly, $T_1(I-T_2T_2^*)T_1^*\le I- T_2T_2^*$ implies that
\[
G_2 - T_1 G_2 T_1^* \geq 0,
\]
and hence the claim follows.

In this context we remark, in view of $C_{T_i} C_{T_j} = C_{T_j} C_{T_i}$ for all $i=1, \ldots, n$, that (see Definition \ref{defn-PT})
\[
S_T(G_i) = \sum_{\bm{k} \in \Z^{n-2}} (-1)^{|\bm{k}|}
\h_{i,n}^{\bm{k}} G_i \h_{i,n}^{*\bm{k}},
\]
where $\h_{i,n}=(T_1,\dots,T_{i-1}, T_{i+1},\dots,T_{n-1})
\in \T^{(n-2)}(\clh)$ for $1\leqslant i\leqslant n-1$. In particular, if $n = 4$, then
\[
\prod_{\stackrel{j=1}{j\ne 3}}^{3}(I_{\mathcal{B}(\Hil)}-
C_{T_j})(G_3) = G_3 - T_1 G_3 T_1^* - T_2 G_3 T_2^* + T_1 T_2 G_3
T_1^* T_2^*.
\]
Moreover, $(T_1, T_2, T_3) \in \T^3(\clh)$ if there exist positive operators $G_1$ and $G_2$ in $\clb(\clh)$ such that
\[
I - T_3 T_3^* = G_1 + G_2,
\]
and
\[
G_2 - T_1 G_2 T^*_1 \geq 0 \quad \mbox{and} \quad G_1 - T_2 G_1 T^*_2 \geq 0.
\]

\begin{rem}\label{remark-compare}
Let $T = (T_1,\dots,T_n)\in \T^n(\Hil)$ and let $1 \leq p < q \leq n$. Recall from Subsection 2.3 in \cite{BDHS} that $T \in \T^n_{p,q}(\clh)$ if $\h_p \in \mathbb{S}_{n-1}(\clh)$ and $\h_q$ satisfies the Szeg\"{o} positivity. We claim that 
\[
\T^n_{p,q}(\clh) \subseteq \clp_n(\clh).
\]
Indeed, for $T \in \T^n(\clh)$, without any loss of generality, we assume that $\hat{T}_n\in \mathbb{S}_{n-1}(\clh) $ and $\hat{T}_1$ satisfies the Szeg\"{o} positivity. Then we set $G_1 = I-T_nT_n^*$ and $G_i = 0$ for all $i = 2, \ldots, n-1$, and consequently
\[
\prod_{j=2}^{n-1}(I_{\mathcal{B}(\Hil)}-
C_{T_j})(G_1)= \mathbb{S}_{n-1}^{-1}(\h_1, \h_1^*) \ge 0,
\]
as $\hat{T_1}\in\mathbb{S}_{n-1}(\clh)$. Therefore $T\in\p_{n}(\Hil)$, and hence, $\p_n(\clh)$ is considerably larger than $\T^n_{p,q}(\clh)$. %Therefore, in the present paper, Theorem \ref{dilation2} gives a different construction of
% (the commutant lifting theorem and)
%isometric dilations for tuples in $\T^n_{p,q}(\clh)$ which are of finite rank.
\end{rem}

\newsection{Preparatory Results}\label{section-preparatory}

This section sets up some of the needed terminology and isolates some preparatory results. We start by considering the Hardy space over the unit polydisc.

The Hardy space $H^2(\D^n)$ over $\D^n$ is the Hilbert space of all
analytic functions $f = \sum_{\bm{k} \in \Z^n} a_{\bm{k}}
z^{\bm{k}}$ on $\D^n$ such that
\[
\|f\| = (\sum_{\bm{k} \in \Z^n} |a_{\bm{k}}|^2)^{\frac{1}{2}} < \infty.
\]
Moreover, for a Hilbert space $\cle$, the $\cle$-valued Hardy space
on $\D^n$ is denoted by $H^2_{\cle}(\D^n)$. We will also identify
$H^2_{\cle}(\D^n)$ with $H^2(\D^n) \otimes \cle$ via the unitary map
$\bm{z}^{\bm{k}} \eta \mapsto \bm{z}^{\bm{k}}\otimes \eta$ for all
$\bm{k} \in \Z^n$ and $\eta \in \cle$. It is well known that the
$\clb(\cle)$-valued function
\[
(\z, \w) \in \D^n \times \D^n \raro \mathbb{S}_n(\z, \w) I_{\cle}
\]
is the reproducing kernel for $H^2_{\cle}(\D^n)$, where
\[
\mathbb{S}_n(\z,\w)= \prod_{i=1}^n(1-z_i\bar{w_i})^{-1} \quad \quad
(\z, \w \in \D^n),
\]
is the Szeg\"{o} kernel on $\D^n$. Let $(M_{z_1}, \ldots, M_{z_n})$
denote the $n$-tuple of multiplication operators on
$H^2_{\cle}(\D^n)$. Here $(M_{z_i} f)(\w) := w_i f(\w)$
for all $f \in H^2_{\cle}(\D^n)$, $\w \in \D^n$ and $i = 1, \ldots,
n$. It follows from the definition that
\[
M_{z_i}^* M_{z_i} = I_{H^2_{\cle}(\D^n)}, \quad M_{z_i} M_{z_j} = M_{z_j} M_{z_i} \quad \mbox{and} \quad M_{z_p}^* M_{z_q} = M_{z_q} M_{z_p}^*,
\]
for all $i, j = 1, \ldots, n$ and $1 \leq p < q \leq n$.

Now we recall the fractional linear transformation representations of Schur-Agler class of functions. Let $\clh_i$, $i = 1, \ldots, n$, $\cle$ and $\cle_*$ be Hilbert spaces, and let
\[
E(\z) = \mathop{\oplus}^n_{i=1} z_iI_{\Hil_i} \in \clb(\mathop{\oplus}_{i=1}^n \Hil_i),
\]
the diagonal operator for all $\z \in \D^n$. Let
\begin{equation}\label{eq-6.1}
U=\begin{bmatrix}A &B\\ C& D\end{bmatrix} : \cle \oplus (\mathop{\oplus}_{i=1}^n \Hil_i) \raro
\cle_* \oplus (\mathop{\oplus}_{i=1}^n \Hil_i),
\end{equation}
be a unitary operator (known as \textit{colligation matrix}). Then the \textit{transfer function} $\tau_{U}$ corresponding to $U$
is defined by
\begin{equation}\label{eq-6.2}
\tau_{U}(\z) = A + B E(\z) (I_{\Hil}- D E(\z))^{-1}C \quad \quad (\z\in\D^n).
\end{equation}
Since $\|D\| \leq 1$, and so $\|D E(\z) \| <
1$ for all $\z \in \D^n$, it follows that $\tau_U$ is a
$\mathcal{B}(\cle, \cle_*)$-valued analytic function on $\D^n$.
Moreover, a standard and well-known computation (cf. \cite{AEM},
\cite{BLTT}) implies that
\begin{equation}\label{identity}
I - \tau_U(\z)^* \tau_{U}(\z) = C^* (I_{\Hil} - E(\z)^* D^*)^{-1}
(I_{\Hil}- E(\z)^* E(\z)) (I_{\Hil} - D E(\z))^{-1}C,
\end{equation}
for all $\z \in \D^n$. We conclude that $\tau_U$ is a contractive
multiplier, that is, $\tau_U \in H^\infty_{\clb(\cle,
\cle_*)}(\D^n)$ and $\|M_{\tau_U}\| \leq 1$ where
$M_{\tau_U} : H^2_{\cle}(\D^n) \raro H^2_{\cle_*}(\D^n)$ is the multiplication operator defined by 
\[
M_{\tau_U} f = \tau_U f \quad \quad (f \in H^2_{\cle}(\D^n)).
\]
The celebrated realization theorem of Agler states the following: $\Phi \in
\SA_n(\cle, \cle_*)$ (see Section \ref{sec-intro} for the definition of $\SA_n(\cle, \cle_*)$) if and only if there exist Hilbert spaces
$\{\clh_i\}_{i=1}^n$ and a unitary colligation matrix $U$, as in \eqref{eq-6.1}, such that $\Phi = \tau_U$.

The following basic result will be useful. Part of the proof relies on the following general result: Let $A$ be a contraction on a Hilbert space $\clk$, $\lambda \in \mathbb{T}$ and let $f \in \clk$. Then
\[
A f = \lambda f \Leftrightarrow A^* f = \bar{\lambda} f.
\]

\begin{lemma}\label{transfer}
Let $\clh_1, \ldots, \clh_n$ and $\cle$ be finite dimensional
Hilbert spaces, and let
\[
U=\begin{bmatrix}A &B\\ C& D\end{bmatrix} \in \clb(\cle \oplus (\mathop{\oplus}_{i=1}^n \Hil_i)),
\]
be a unitary matrix. Then the transfer function $\tau_{U}$ defined by
\[
\tau_{U}(\z) = A + B E(\z) (I_{\Hil}-D E(\z))^{-1}C \quad \quad (\z
\in \D^n),
\]
is unitary-valued a.e. on $\mathbb{T}^n$. Moreover, if $A \in \clb(\cle)$ is a completely non-unitary, then for each $\z \in \D^n$ the operator $\tau_{U}(\z)$ does not have eigenvalues on the unit circle.
\end{lemma}
\begin{proof}
Clearly, $\z \mapsto \det (I_{\Hil}-D E(\z))$ is a non-vanishing polynomial in $\D^n$ and hence 
\[
\z \mapsto (I_{\Hil}-D E(\z))^{-1},
\]
is a non-zero rational function in $H^\infty(\D^n)$. This means that $\z\mapsto \det(I_{\Hil}-DE(\z))$ does not vanish on a set of positive measure on $\mathbb{T}^n$ and hence $\z \mapsto (I-D E(\z))^{-1}$ exists a.e. on $\mathbb{T}^n$. The first part now follows from \eqref{identity}.

\NI For the second part, let $\z \in \D^n$, $\lambda \in \mathbb{T}$ and let $f$ be a non-zero element in $\cle$. Suppose
\[
\tau_U(\z) f = \lambda f.
\]
Since $\tau_U(\z)$ is a contraction, it follows that
\[
\tau_U(\z)^* f = \bar{\lambda} f,
\]
and so
\[
(I_{\cle} - \tau_U(\z)^* \tau_{U}(\z)) f = 0.
\]
It follows easily from \eqref{identity} that $C f = 0$. Then $A f = \tau_U(\z) f = \lambda f$ and so
\[
A^* f = \bar{\lambda} f.
\]
This implies that $A$ has a non-trivial unitary part, which leads to a contradiction.
\end{proof}

Let us now proceed to set the stage of the technical part in the proofs of our main results. Let $\cle$ be a Hilbert space. Define the evaluation map $ev_0 : H^2_{\cle}(\D^n) \raro \cle$ by
\[
ev_0f = f(\bm{0}) \quad \quad (f \in H^2_{\cle}(\D^n)).
\]
Then $ev_0^* : \cle \raro H^2_{\cle}(\D^n)$ sends $\eta \in \cle$ to the constant function $\eta \in H^2_{\cle}(\D^n)$ in the following sense:
\[
(ev^*_0 \eta)(\z) = \eta \quad \quad  (\z\in \D^n).
\]

Now let $T\in\T^n(\clh)$ and suppose $\h_n\in \sz_{n-1}(\clh)$. Let $\{F_1, \ldots, F_{n-1}\} \subseteq \clb(\Hil)$ and $\clk$ be a Hilbert space. For each $i=1,\ldots,n-1$, set
\[
\clf_i = \overline{\mbox{ran}} F_i,
\]
and define 
\[
\clf = \Big(\bigoplus_{i=1}^{n-2}\clf_i \Big) \oplus \Big( \clf_{n-1} \oplus \clk \Big),
\]
the $(n-1)$-fold direct sum of Hilbert spaces. Define 
\[
D_{\h_n} = \Big(\mathbb{S}_{n-1}^{-1}(\h_n, \h_n^*)\Big)^{\frac{1}{2}} \quad \mbox{and} \quad D_{\h_n} = \overline{\mbox{ran}}D_{\h_n}.
\]
Let
\begin{equation}\label{unitary}
U =\begin{bmatrix}A &B\\C&D\end{bmatrix}:
\cld_{\h_n} \oplus \clf \raro \cld_{\h_n} \oplus \clf,
\end{equation}
be a unitary operator such that
\begin{equation}
\label{generating identity}
U(D_{\h_n}h,F_1T_1^*h,\ldots, F_{n-2}T_{n-2}^*h, (F_{n-1}T_{n-1}^*h, 0_{\clk})) = (D_{\h_n}T_n^*h,F_1h,\ldots, F_{n-2}h, (F_{n-1}h,0_{\clk})),
\end{equation}
for all $h \in \clh$. Define $\imath$ and $Y$ in $\clb(\clh, \clf)$ by
\begin{equation}\label{eq-imath}
\imath h = (F_1h,\ldots, F_{n-2}h, (F_{n-1}h, 0_{\clk})),
\end{equation}
and
\begin{equation}\label{eq-Y}
Y h = (F_1T_1^*h,\ldots, F_{n-2}T_{n-2}^*h, (F_{n-1}T_{n-1}^*h,0_{\clk})),
\end{equation}
for all $h\in\Hil$. Then \eqref{generating identity} yields
\[
\begin{bmatrix} A & B\\C & D \end{bmatrix} \begin{bmatrix} D_{\h_n} h \\ Y h \end{bmatrix} = \begin{bmatrix} D_{\h_n} T^*_n h \\ \imath h \end{bmatrix}.
\]
Therefore
\begin{equation}\label{id1}
D_{\h_n}T_n^*=AD_{\h_n}+BY \quad \mbox{and} \quad \imath=CD_{\h_n}+DY.
\end{equation}
Suppose $\Phi = \tau_{U^*}$, that is (see \eqref{eq-6.2})
\begin{equation}\label{Phi}
\Phi(\z)=A^*+C^*(I_\clf-E(\z)D^*)^{-1}E(\z)B^*\quad \quad (\z \in \D^{n-1}).
\end{equation}
Define $\Psi: \clf\to H^2_{\cld_{\h_n}}(\D^{n-1})$ by
\begin{equation}\label{Psi}
[\Psi x](\z)=C^*(I_{\clf} - E(z)D^*)^{-1}x \quad \quad (x \in \clf, \z \in \D^{n-1}),
\end{equation}
and set
\[
\Psi=[\Psi_1\ldots\Psi_{n-1}],
\]
where, $\Psi_i: \clf_i\to H^2_{\cld_{\h_n}}(\D^{n-1})$ for all $i=1,\ldots,n-2${\tiny }, and $\Psi_{n-1}: \clf_{n-1}\oplus\clk\to H^2_{\cld_{\h_n}}(\D^{n-1})$. It is convenient to represent $D$ and $B$ as row operators:
\[
D = \begin{bmatrix} D_1 & \cdots & D_{n-2} & D_{n-1} \end{bmatrix} : \Big(\bigoplus_{i=1}^{n-2}\clf_i \Big) \oplus \Big( \clf_{n-1} \oplus \clk \Big) \raro \clf,
\]
where $D_i = D|_{\clf_i} : \clf_i \raro \clf$ for all $i=1, \ldots, n-2$ and $D_{n-1} = D|_{\clf_{n-1} \oplus \clk}$, and similarly
\[
B = \begin{bmatrix} B_1 & \cdots & B_{n-2} & B_{n-1} \end{bmatrix} : \Big(\bigoplus_{i=1}^{n-2}\clf_i \Big) \oplus \Big( \clf_{n-1} \oplus \clk \Big) \raro \cld_{\h_n}.
\]
The following two lemmas will be used in Section \ref{section-dilation} to prove the dilation theorem. We closely follow the ideas of \cite{BLTT}.

\begin{lemma}\label{pro-psi}
Let $\Phi$ and $\Psi$ be as above. Then $\Psi$ is a contraction. Moreover
\[
\Psi=ev_0^* C^*+\sum_{i=1}^{n-1}M_{z_j}\Psi_jD_j^*,
\]
and
\[
M_{\Phi} ev_0^* = ev_0^* A^*+\sum_{j=1}^{n-1}M_{z_j}\Psi_jB_j^*.
\]
\end{lemma}
\begin{proof}
It follows directly from Lemma 3.2 of \cite{BLTT} that $\Psi$ is a contraction. For the second part, let $x=(x_1,\ldots,x_{n-1})\in\clf$ (note that $x_{n-1} \in \clf_{n-1} \oplus \clk$) and $\z=(z_1,\ldots,z_{n-1})\in\D^{n-1}$. Then
\[
\Psi E(\z)x=\sum_{j=1}^{n-1}\Psi_jz_jx_j=\sum_{j=1}^{n-1}z_j\Psi_jx_j,
\]
and hence
\begin{align*}
[\Psi x](\z) & = C^*(I-E(z)D^*)^{-1}x
\\
& = C^*(I-E(\z)D^*)^{-1}(I-E(\z)D^*+E(\z)D^*)x
\\
& = C^*x + \Big(C^* (I-E(\z)D^*)^{-1} \Big)( E(\z)D^* x)
\\
& = (ev_0^* C^*x)(\z)+[\Psi E(\z)D^*x](\z)
\\
& = (ev_0^* C^*x)(\z) + \big(\sum_{j=1}^{n-1}z_j \Psi_j D_j^* x\Big)(\z)
\\
& = (ev_0^* C^*x)(\z) + \sum_{j=1}^{n-1}z_j \big(\Psi_j D_j^* x\Big)(\z)
\\
& = (ev_0^* C^*x)(\z) + \sum_{j=1}^{n-1}\big( M_{z_j} \Psi_j D_j^* x\Big)(\z).
\end{align*}
This implies that $\Psi=e_0^* C^*+\sum_{i=1}^{n-1}M_{z_j}\Psi_jD_j^*$. On the other hand, for each $\eta \in \cld_{\h_n}$, we have
\[
M_{\Phi} ev_0^* \eta = ev_0^* A^* \eta +\Psi E(\z)B^* \eta = ev_0^* A^* \eta + \Big[\sum_{j=1}^{n-1}M_{z_j}\Psi_jB_j^*\Big](\eta).
\]
This completes the proof.	
\end{proof}

For each $X \in \clb(\clh)$ and natural number $m$ we define (see page 656 in \cite{BLTT}) a positive map
\[
\varSigma_{X}^m : \clb(\clh)\to\clb(\clh),
\]
as
\[
\varSigma_{X}^m(A)=\sum_{k=0}^{m-1}X^k A X^{*k} \quad \quad (A \in\clb(\clh)).
\]
Let $T=(T_1,\ldots,T_n)$ be an $n$-tuple of commuting operators on $\clh$. It is easy to see that
\[
\varSigma_{T_i}^m \varSigma_{T_j}^m = \varSigma_{T_j}^m \varSigma_{T_i}^m,
\]
for all $i,j = 1, \ldots, n$. Define
\[
\varSigma_T^m(A) = \Big(\prod_{j=1}^{n}\varSigma_{T_j}^m \Big) (A) \quad \quad (A\in\clb(\clh)).
\]
Clearly, if $A\ge0$, then $\{\varSigma_T^m(A)\}_{m=1}^{\infty}$ is an increasing sequence of positive operators and we set
\[
\varSigma_T(A) = SOT-\lim_{m \raro \infty} \varSigma_T^m(A),
\]
provided the limit exists. Moreover, for each $i, j=1,\ldots, n$, and $m \geq 1$, we have
\[
\varSigma_{T_i}^m C_{T_j} = C_{T_j} \varSigma_{T_i}^m,
\]
as $T$ is a commuting tuple (see \eqref{Cmap} for the definition of the conjugate map $C_{T_j}$). Therefore
\begin{equation}\label{eq-10.1}
\varSigma_{T_i}^m (I - C_{T_i}) = (I - C_{T_i}) \varSigma_{T_i}^m = (I - C_{T^m_i}),
\end{equation}
for all $i=1,\ldots, n$, and hence
\begin{equation}\label{eq-sigma and C}
\varSigma_T^m \Big(\prod_{{j=1}}^{n}(I_{\mathcal{B}(\Hil)}-
C_{T_j}) \Big) = \Big( \prod_{{j=1}}^{n}(I_{\mathcal{B}(\Hil)}-
C_{T_j^m})\Big).
\end{equation}
Finally, it will be clear from the context in what follows that given a Hilbert space $\cle$, $M_z$ will denote the $n$-tuple of multiplication operators $(M_{z_1}, \ldots, M_{z_n})$ on $H^2_{\cle}(\D^n)$. Consequently, $\hat{M}_{z_j}$ will denote the $(n-1)$-tuple 
\[
(M_{z_1}, \dots, M_{z_{j-1}}, M_{z_{j+1}},\dots, M_{z_n}),
\]
obtained from $M_z$ by removing $M_{z_j}$.

We now prove a technical lemma.

\begin{lemma}\label{psipsi*}
Let $\clh_1$, \ldots, $\clh_n$ and $\cle$ be Hilbert spaces, $\Psi = [\Psi_1\cdots\Psi_n]: {\displaystyle}\oplus_{i=1}^n\clh_i\to H^2_{\cle}(\D^n)$ a bounded linear operator and let $\Phi \in H^\infty_{\clb(\cle)}(\D^n)$. If
\[
\|(ev_0 f, \Psi_1^*M_{z_1}^*f,\ldots,\Psi_n^*M_{z_n}^*f)\|_{\cle \oplus ({\oplus}_{i=1}^n \clh_i)} =\|(ev_0 M_{\Phi}^*f,\Psi_1^*f,\ldots,\Psi_n^*f)\|_{\cle \oplus (\oplus_{i=1}^n \clh_i)},
\]
for all $f\in H^2_{\cle}(\D^n)$, then
\[
\varSigma_{\hat{M}_{z_j}}(\Psi_j\Psi_j^*)<\infty
\]
for all $j=1,\ldots,n$.
\end{lemma}

\begin{proof}
The norm equality condition implies that
\[
ev_0^* ev_0 + \sum_{j=1}^{n}M_{z_j}\Psi_j\Psi_j^*M_{z_j}^* = M_{\Phi} e_0^* e_0 M_{\Phi}^* + \sum_{j=1}^{n}\Psi_j \Psi_j^*,
\]
and so
\begin{align*}
\sum_{j=1}^{n}\Psi_j \Psi_j^*-\sum_{j=1}^{n}M_{z_j}\Psi_j\Psi_j^*M_{z_j}^* & = ev_0^* ev_0 -M_{\Phi} ev_0^* ev_0 M_{\Phi}^*
\\
& =\Big(\prod_{{j=1}}^{n}(I-
C_{M_{z_j}})\Big)(I) - M_{\Phi} \Big(\big(\prod_{{j=1}}^{n}(I-C_{M_{z_j}})\big)(I) \Big) M_{\Phi}^*.
\end{align*}
This gives
\[
\sum_{j=1}^{n}(I-C_{M_{z_j}})(\Psi_j\Psi_j^*) = \Big( \prod_{{j=1}}^{n}(I-C_{M_{z_j}}) \Big)(I-M_{\Phi}M_{\Phi}^*),
\]	
as $M_{\Phi} M_{z_i} = M_{z_i} M_{\Phi}$ for all $i=1, \ldots, n$. Applying $\varSigma_{M_{\z}}^m$ to both sides of the above equality and then using \eqref{eq-10.1} and \eqref{eq-sigma and C} we obtain
\[
\sum_{j=1}^{n}\varSigma_{\hat{M}_{z_j}}^m \Big((I-
C_{M_{z_j}^m})(\Psi_j\Psi_j^*)\Big) = \prod_{{j=1}}^{n}(I-
C_{M_{z_j}^m})(I-M_{\Phi} M_{\Phi}^*).
\]
If we set
\[
\clk_0= \mathop{\cup}_{k=0}^{\infty} \mathop{\cap}_{j=1}^{n}\text{Ker}\,M_{z_j}^{*k},
\]
then for any $f\in\clk_0$ we have
\[
\sum_{j=1}^{n}\varSigma_{\hat{M}_{z_j}}^m(\Psi_j\Psi_j^*)f=(I-M_{\Phi}M_{\Phi}^*)f
\]
for sufficiently large $m$. Since $\clk_0$ is dense in $H^2_{\cle}(\D^n)$, it follows that $\varSigma_{\hat{M}_{z_j}}(\Psi_j\Psi_j^*)<\infty$ for all $j=1,\ldots,n$.
\end{proof}

\newsection{Dilations of tuples in $\p_n(\clh)$}\label{section-dilation}

The purpose of this section is to construct explicit liftings and isometric dilations of finite rank tuples in $\p_n(\clh)$. Note that if $V \in \T^n (\clk)$ is an isometric dilation of $T \in \T^n(\clh)$ with $\Pi : \clh \raro \clk$ as the corresponding intertwining isometry (see the definition of dilations in Section \ref{sec-intro}) and
\[
\clq = \textit{ran~} \Pi,
\]
then $\clq$ is a joint invariant
subspace for $(V_1^*, \ldots, V_n^*)$. Moreover, $(T_1, \ldots, T_n)$ on
$\clh$ and $(P_{\clq} V_1|_{\clq}, \ldots, P_{\clq} V_n|_{\clq})$ on
$\clq$ are unitarily equivalent (via the unitary $\Pi : \clh \raro
\clq$), and
\[
(P_{\clq} V|_{\clq})^{* \bm{k}} = V^{* \bm{k}}|_{\clq},
\]
for all $\bm{k} \in \Z^n$.

The following dilation result is well known (see \cite{CV1}, \cite{MV} and also see \cite{BDHS}):

\begin{thm}\label{th-dc dil}
Let $T \in \mathbb{S}_n(\clh)$. Suppose $\Dt = \mathbb{S}_n^{-1}(T, T^*)^{1/2}$ and $\dt
=\overline{\text{ran}}~\Dt$. Then $\Pi: \Hil\to H^2_{\cld_T}(\D^n)$,
defined by
\[
(\Pi h)(\z) = \sum_{\bm{k} \in \Z^n} {\bm{z}}^{\bm{k}} D_T
T^{*\bm{k}} h \quad \quad \quad  (h \in \clh, \z \in \D^n),
\]
is an isometry and satisfies $\Pi T_i^* = M_{z_i}^*\Pi$ for all $i = 1,
\ldots, n$. Moreover, the dilation is minimal, that is
\[
H^2_{\cld_T}(\D^n) = \overline{span} \{\z^{\bm{k}} (\Pi \clh): \bm{k} \in \Z^n\}.
\]
\end{thm}

We call the map $\Pi$ as the \textit{canonical dilation} of $T$. The following standard lemma will be useful.

\begin{lemma}\label{l2}
Let $T \in \mathbb{S}_n(\clh)$ and let $\Pi$ be the canonical dilation of $T$.
 Then for all $\bm{k} \in \Z^n$ and $\eta \in \cld_{T}$, we have
\[
\Pi^*(\z^{\bm{k}}\otimes \eta)=T^{\bm{k}}D_{T}\eta.
\]
\end{lemma}
\begin{proof}
For $h\in \Hil$, it is easy to see that
\[
ev_0 \Pi h = ev_0 \big(\sum_{\bm{l} \in \Z^{n-1}} {\z}^{\bm{l}} \otimes D_{T} T^{*\bm{l}} h\big) = D_{T}h,
\]
thus
\begin{equation}\label{Q*pi}
ev_0 \Pi=D_{T}.
\end{equation}
Hence, for $\eta \in \cld_{T}$ and $\bm{k}\in\Z^{n-1}$ we obtain
\[
\Pi^*(\z^{\bm{k}}\otimes \eta)=\Pi^*M_{\z}^{\bm{k}} ev_0^* \eta =T^{\bm{k}}\Pi^* ev_0^* \eta = T^{\bm{k}}D_{T} \eta,
\]
and the proof follows.
\end{proof}

We are now ready to prove the first commutant lifting theorem of this paper. The explicit representation of the lifting $\Phi\in \SA_{n-1}(\cld_{\h_n},\cld_{\h_n})$ below will be useful in our isometric dilations and von Neumann inequality. For a $T\in \T^n(\clh)$ and $1\le j\le n-1$, recall that, $\hat{T}_{j,n}$ denotes the $(n-2)$-tuple
\[
(T_1,\dots,T_{j-1},T_{j+1},\dots, T_{n-1}),
\]
obtained from $T$ by removing $T_j$ and $T_n$

\begin{thm}\label{key theorem}
Let $\clh$ and $\clk$ be Hilbert spaces, $T\in \T^n(\clh)$, and let $\{F_1,\ldots,F_{n-1}\} \subseteq \clb(\Hil)$. Suppose $\h_n \in \mathbb{S}_{n-1}(\clh)$, $\clf_i=\overline{\mbox{ran}} F_i$ for all $i=1,\dots, n-1$, $\Pi:\Hil\to H^2_{\cld_{\h_n}}(\D^{n-1})$ is the canonical dilation of $\h_n$, and
\[
\clf = \Big(\bigoplus_{i=1}^{n-2} \clf_i \Big) \oplus \Big(\clf_{n-1} \oplus \clk \Big).
\]
If $\varSigma_{\hat{T}_{i,n}}(F_i^*F_i)$ exists for all $i=1,\dots, n-1$, and
 $U: \cld_{\h_n}\oplus\clf \raro \cld_{\h_n}\oplus \clf$ is a unitary satisfying
\[
U(D_{\h_n}h,F_1T_1^*h,\ldots, F_{n-2}T_{n-2}^*h, (F_{n-1}T_{n-1}^*h,0_{\clk})) = (D_{\h_n}T_n^*h,F_1h,\ldots, F_{n-2}h, (F_{n-1}h,0_{\clk})),
\]
for all $h\in\Hil$, then
\[
\Pi T_n^* = M_{\Phi}^*\Pi,
\]
where $\Phi\in \SA_{n-1}(\cld_{\h_n},\cld_{\h_n})$ is the transfer function of $U^*$.
\end{thm}

\begin{proof}
Let $U=\begin{bmatrix}A &B\\C&D\end{bmatrix}$ be the block matrix representation of $U$ with respect to the decomposition $D_{\h_n}\oplus \clf$. Suppose $\Phi = \tau_{U^*}$ (see \eqref{Phi}), $\Psi$ be as in (\ref{Psi}) and let $f\in H^2_{\cld_{\h_n}}(\D^{n-1})$. Then, in view of \eqref{generating identity}, Lemma \ref{pro-psi} implies that
\[
U(ev_0 f, \Psi_1^*M_{z_1}^*f,\ldots, \Psi_{n-1}^* M_{z_{n-1}}^*f) = (ev_0 M_{\Phi}^*f,\Psi_1^*f,\ldots, \Psi_{n-1}^*f).
\]
In particular
\[
\|(ev_0 f, \Psi_1^*M_{z_1}^*f,\ldots, \Psi_{n-1}^* M_{z_{n-1}}^*f)\|=\|(ev_0 M_{\Phi}^*f,\Psi_1^*f,\ldots, \Psi_{n-1}^*f)\|,
\]
and therefore
\[
\varSigma_{\hat{M}_{z_j}}(\Psi_j\Psi_j^*)<\infty,
\]
for all $j=1,\ldots,n-1$, by Lemma \ref{psipsi*}. Hence
\[
\varSigma_{\hat{T}_{j, n}}(\Pi^*\Psi_j\Psi_j^*\Pi)=\Pi^* \Big(\varSigma_{\hat{M}_{z_j}}(\Psi_j\Psi_j^*)\Big)\Pi < \infty,
\]
for all $j=1,\ldots,n-1$. Now define $\imath$ and $Y$ in $\clb(\clh, \clf)$ as (see \eqref{eq-imath} and \eqref{eq-Y})
\[
\imath h = (F_1h,\ldots, F_{n-2}h, (F_{n-1}h,0_{\clk})),
\]
and
\[
Y h = (F_1T_1^*h,\ldots, F_{n-2}T_{n-2}^*h, (F_{n-1}T_{n-1}^*h, 0_{\clk})),
\]
for all $h \in \clh$. Also define $\Gamma: \clf\to \clh $ by
\[
\Gamma=\imath^*-\Pi^*\Psi.
\]
We set $\Gamma=[\Gamma_1 \cdots \Gamma_{n-1}]$, where
\[
\Gamma_i = \Gamma |_{\clf_i}=F_i^*|_{\clf_i}-\Pi^*\Psi_i : \clf_i \raro \clh,\]
for all $i = 1,\dots, n-2$, and
 \[
\Gamma_{n-1}=\Gamma |_{\clf_{n-1}\oplus \clk}=
\begin{bmatrix}F_{n-1}^*|_{\clf_{n-1}} & 0\end{bmatrix} -\Pi^*\Psi_{n-1}: \clf_{n-1} \oplus \clk \raro \clh.
\]
Now for each $j=1,\dots,n-2$, we have
\[
\begin{split}
\Gamma_j\Gamma_j^*& = (F_j^*|_{\clf_j}-\Pi^*\Psi_j)(F_j-\Psi_j^*\Pi)
\\
& \leq (F_j^*|_{\clf_j}-\Pi^*\Psi_j)(F_j-\Psi_j^*\Pi)+(F_j^*|_{\clf_j}+\Pi^*\Psi_j)(F_j+\Psi_j^*\Pi)
\\
& = 2 (F_j^*F_j+\Pi^*\Psi_j\Psi_j^*\Pi),
\end{split}
\]
and hence
\[
\varSigma_{\hat{T}_{j,n}}(\Gamma_j\Gamma_j^*)<\infty.
\]
By a similar computation, we also have that $\varSigma_{\hat{T}_{n-1, n}}(\Gamma_{j}\Gamma_j^*)<\infty$. On the other hand, since (see \eqref{id1})
\[
\imath = C D_{\h_n} + DY,
\]
we have
\begin{align*}
\Gamma=& \imath^*-\Pi^*\Psi
\\
= & D_{\h_n}C^*+Y^*D^*-\Pi^*QC^*-\Pi^*\sum_{j=1}^{n-1}M_{z_j}\Psi_jD_j^*
\\
=&D_{\h_n}C^*+\sum_{j=1}^{n-1}T_jF_j^*D_j^*-D_{\h_n}C^*-\sum_{j=1}^{n-1}T_{j}\Pi^*\Psi_jD_j^*
\\
=&\sum_{j=1}^{n-1}T_j\Gamma_jD_j^*\\
=&\begin{bmatrix} T_1\Gamma_1 & \cdots & T_{n-1}\Gamma_{n-1}\end{bmatrix} \begin{bmatrix}
D_1^*\\
\vdots\\
D_{n-1}^*
\end{bmatrix},
\end{align*}
and hence
\[
\Gamma\Gamma^*= \begin{bmatrix} T_1\Gamma_1 & \cdots & T_{n-1}\Gamma_{n-1} \end{bmatrix} D^*D \begin{bmatrix}
\Gamma_1^*T_1^*\\
\vdots\\
\Gamma_{n-1}^*T_{n-1}^*
\end{bmatrix}\leq\sum_{j=1}^{n-1}T_j\Gamma_j\Gamma_j^*T_j^*,
\]
as $D$ is a contraction. This implies
\[
\sum_{j=1}^{n-1}(I_{\mathcal{B}(\Hil)}-
C_{T_j})(\Gamma_j\Gamma_j^*)=\sum_{j=1}^{n-1}\Gamma_j\Gamma_j^*-\sum_{j=1}^{n-1}T_j\Gamma_j\Gamma_j^*T_j^*\leq 0.
\]
Next, for each natural number $m$ and $j = 1, \ldots, n-1$, we have
\[
\begin{split}
\varSigma_{\hat{T}_{j,n}}^m (\Gamma_j\Gamma_j^*)-T_j^{m}\varSigma_{\hat{T}_{j,n}}(\Gamma_j\Gamma_j^*)T_j^{*m} & \leq \varSigma_{\hat{T}_{j,n}}^m(\Gamma_j\Gamma_j^*)-T_j^{m}\varSigma_{\hat{T}_{j,n}}^m(\Gamma_j\Gamma_j^*)
T_j^{*m}
\\
& = \varSigma_{\hat{T}_{j,n}}^m(I_{\mathcal{B}(\Hil)}-
C_{T^m_j})(\Gamma_j\Gamma_j^*),
\end{split}
\]
and since (see \eqref{eq-10.1})
\[\varSigma_{\hat{T}_{j,n}}^m(I_{\mathcal{B}(\Hil)}- C_{T^m_j})=\varSigma_{\hat{T}_{j,n}}^m\varSigma_{T_j}^m(I_{\mathcal{B}(\Hil)}-
C_{T_j})=\varSigma^m_{\hat{T}_n}(I_{\mathcal{B}(\Hil)}- C_{T_j}),
\]
it follows that
\[
\sum_{j=1}^{n-1}\big[\varSigma_{\hat{T}_{j,n}}^m (\Gamma_j\Gamma_j^*)-T_j^{m}\varSigma_{\hat{T}_{j,n}}(\Gamma_j\Gamma_j^*)T_j^{*m}\big]
 \leq \varSigma_{\hat{T}_n}^m\big(\sum_{j=1}^{n-1}(I_{\mathcal{B}(\Hil)}-
C_{T_j})(\Gamma_j\Gamma_j^*)\big) \leq 0.
\]
Since $T_j$ is pure, passing to the limit as $m \to\infty$, we get
\[
\sum_{j=1}^{n-1}\varSigma_{\hat{T}_{j,n}}(\Gamma_j\Gamma_j^*)\leq 0.
\]
On the other hand, since $\varSigma_{\hat{T}_{j,n}}(\Gamma_j\Gamma_j^*)\geqslant 0$, by definition, we have that
\[
\varSigma_{\hat{T}_{j,n}}(\Gamma_j\Gamma_j^*)=0,
\]
for all $j=1,\dots,n-1$. Hence
\[
 F_j^*|_{\clf_j}-\Pi^*\Psi_j=\Gamma_j= 0\quad \text{and }
 \quad \tilde{F}_{n-1}^*-\Pi^*\Psi_{n-1}=\Gamma_{n-1}=0,
\]
$j=1,\ldots,n-2$. Finally, for $\bm{p}\in \Z^{n-1} $ and $\eta \in\cld_{\h_n}$, we have
\[
\Pi^*M_{\Phi}(\z^{\bm{p}}\otimes \eta)
= \Pi^*M_{\z^{\bm{p}}}M_{\Phi}(1\otimes \eta)= \h_n^{\bm{p}}\Pi^*M_{\Phi}(1\otimes \eta)= \h_n^{\bm{p}}\Pi^*M_{\Phi}Q \eta.
\]
But
\[
\h_n^{\bm{p}}\Pi^*M_{\Phi}Q \eta = \h_n^{\bm{p}}\Pi^*\Big[QA^*+\sum_{j=1}^{n-1}M_{z_j}\Psi_jB_j^*\Big]\eta = \h_n^{\bm{p}}\Big[D_{\h_n}A^*+\sum_{j=1}^{n-1}T_{j}\Pi^*\Psi_jB_j^*\Big]\eta,
\]
by Lemma \ref{pro-psi} and \eqref{Q*pi}, and therefore by \eqref{id1} and Lemma \ref{l2}
\[
\Pi^*M_{\Phi}(\z^{\bm{p}}\otimes \eta)
= \h_n^{\bm{p}}T_nD_{\h_n} \eta = T_n\h_n^{\bm{p}}D_{\h_n} \eta
=T_n\Pi^*(\z^{\bm{p}}\otimes \eta).
\]
But since $\{\z^{\bm{p}}\otimes \eta :\bm{p}\in \Z^{n-1}, m\in\cld_{\h_n}\}$ is a total set in $H^2_{\cld_{\h_n}}(\D^{n-1})$, we conclude
\[
\Pi^*M_{\Phi}=T_n\Pi^*,
\]
which completes the proof of the theorem.
\end{proof}

We are now ready to prove the dilation theorem for commuting tuples in $\p_n(\clh)$ (recall Definitions \ref{defn-PT} and \ref{defn-Finite PT}). Recall from Theorem \ref{th-dc dil} that given a $T \in \p_n(\clh)$, the canonical isometric dilation $\Pi:\Hil\to H^2_{\cld_{\h_n}}(\D^{n-1})$ of the $(n-1)$-tuple $\h_n$ is an isometry and 
\[
\Pi T_i^*=M_{z_i}^*\Pi,
\]
for all $i=1,\dots,n-1$.

\begin{thm}\label{dilation2}
If $T \in \p_n(\clh)$, then
there exists a contractive multiplier $\Phi\in \cls_{n-1}(\cld_{\h_n}, \cld_{\h_n})$
such that
\[
\Pi T_i^* = \left \{\begin{array}{ll}
M_{z_i}^* \Pi \;\; \quad \quad \text{if~} i = 1, \ldots, n-1,\\
M_{\Phi}^* \Pi \quad \quad \; \;\;\text{if~} i = n.
\end{array}
\right.
\]
If, in addition, $T$ is finite rank, then $\Phi$ is an inner function.
In particular, a finite rank $T$ in $\p_n(\clh)$ dilates
to a commuting isometries $(M_{z_1},\dots, M_{z_{n-1}},M_{\Phi})$ on $H^2_{\cld_{\h_n}}(\D^{n-1})$.
\end{thm}
\begin{proof}
Let $\{G_1, \ldots, G_{n-1}\}$ be the positive operators on $\clh$ corresponding to $T \in \p_n(\clh)$. Then
\[
I-T_nT_n^*= G_1+\dots+G_{n-1},
\]
so that
\[
\begin{split}
D_{\h_n}^2-T_nD_{\h_n}^2T_n^* & = \sum_{\bm{k} \in \Z^{n-1}}
(-1)^{|\bm{k}|} \h_{n}^{\bm{k}} (I - T_n T_n^*) \h_{n}^{*\bm{k}}
\\
& = \prod_{j=1}^{n-1}(I_{\mathcal{B}(\Hil)}- C_{T_j})(I-T_nT_n^*)
\\
& = \sum_{i=1}^{n-1}\prod_{j=1}^{n-1}(I_{\mathcal{B}(\Hil)}-
C_{T_j})(G_i)
\\
&=  \sum_{i=1}^{n-1}
\prod_{\stackrel{j=1}{j\ne i}}^{n-1} (I_{\mathcal{B}(\Hil)}-
C_{T_j}) (I_{\mathcal{B}(\Hil)}- C_{T_i})(G_i)
\\
&=\sum_{i=1}^{n-1}\Big(\prod_{\stackrel{j=1}{j\ne i}}^{n-1}
(I_{\mathcal{B}(\Hil)}- C_{T_j})
(G_i)-T_i\Big(\prod_{\stackrel{j=1}{j\ne
i}}^{n-1}(I_{\mathcal{B}(\Hil)}- C_{T_j})(G_i)\Big)T_i^*\Big),
\end{split}
\]
in view of $C_{T_p} C_{T_q} = C_{T_q} C_{T_p}$ for all $p,q = 1, \ldots, n-1$. If we define (see Definition \ref{defn-PT})
\[
S_T(G_i) = \prod_{\stackrel{j=1}{j\ne i}}^{n-1}(I_{\mathcal{B}(\Hil)}-
C_{T_j})(G_i),
\]
then $S_T(G_i) \geq 0$ for all $i=1, \ldots, n-1$, by assumption, and
\[
\begin{split}
D_{\h_n}^2 - T_nD_{\h_n}^2T_n^* =\sum_{i=1}^{n-1}\Big(S_T(G_i) - T_i S_T(G_i) T_i^*\Big).
\end{split}
\]
If we let $F_i^2 = S_T(G_i)$ for all $i=1, \ldots, n-1$, for simplicity, then
\[
D_{\h_n}^2 + \sum_{i=1}^{n-1} T_iF_i^2T_i^* = T_nD_{\h_n}^2T_n^* +
\sum_{i=1}^{n-1}F_i^2,
\]
and so the map
\begin{equation}
\label{un}
U: \{(D_{\h_n} h, F_1T_1^* h,\dots, F_{n-1} T_{n-1}^*h):
h\in\Hil\}\to \{(D_{\h_n} T_n^*h, F_1h,\dots, F_{n-1}h):h\in\Hil\}
\end{equation}
defined by
\[
U (D_{\h_n} h, F_1T_1^*h,\dots, F_{n-1}T_{n-1}^*h) = (D_{\h_n}
T_n^*h, F_1 h,\dots, F_{n-1}h),
\]
for all $h\in\Hil$, is an isometry. Set
\[
\clf_i=\overline{\Ran}F_i,
\]
for all $i=1,\dots,n-1$. Then, by adding an infinite dimensional Hilbert space $\clk$, if necessary, one can find a unitary
\begin{equation}\label{eq-U}
U: \cld_{\h_n}\oplus(\oplus_{i=1}^{n-1}\clf_i)\oplus \clk
\to \cld_{\h_n}\oplus(\oplus_{i=1}^{n-1}\clf_i)\oplus \clk,
\end{equation}
such that
\[
 U (D_{\h_n} h, F_1T_1^*h,\dots, F_{n-1}T_{n-1}^*h, 0_{\clk}) = (D_{\h_n}
T_n^*h, F_1 h,\dots, F_{n-1}h, 0_{\clk}),
\]
for all $h\in\Hil$. Finally, since
\[
\varSigma_{\hat{T}_{i,n}}^N(F_iF_i^*)=\prod_{k\neq i,n}(I_{\mathcal{B}(\Hil)}-
C_{T^N_k})(G_i)\leq G_i
\]
implies that $\varSigma_{\hat{T}_{i,n}}(F_iF_i^*)$ exists for all $i=1,\ldots,n-1$, the first part of the theorem follows from Theorem \ref{key theorem}.

\noindent In addition now assume that $T$ is finite rank. Then $\cld_{\h_n}$ and $\clf_i$, $i = 1, \ldots, n-1$, are
all finite dimensional Hilbert spaces. Then the unitary $U$ in \eqref{un} extends to a
unitary on $\cld_{\h_n}\oplus(\oplus_{i=1}^{n-1}\clf_i)$ which we also denote by $U$ and then by applying Theorem~\ref{key theorem} (with $\clk=0$), we have
\[
M_{\Phi}^*\Pi=\Pi T_n^*,
\]
where $\Phi$ is the transfer function corresponding to $U^*$. In this case, Lemma \ref{transfer} yields that
$\Phi$ is a $\mathcal{B}(D_{\h_n})$-valued inner multiplier.
This completes the proof of the theorem.
\end{proof}

\begin{rem}
It is worth mentioning that the converse of the above theorem
is also true. Indeed, if $\cle$ is a Hilbert space, $\Phi \in \cls_{n-1}(\cle, \cle)$ and $\clq \subseteq H^2_{\cle}(\D^{n-1})$ is a joint invariant subspace for
$(M_{z_1}^*, \ldots, M_{z_{n-1}}^*, M_{\Phi}^*)$, then \cite[Theorem 5.1]{BLTT} implies that
\[
(P_{\clq}M_{z_1}|_{\clq},\dots, P_{\clq}M_{z_{n-1}}|_{\clq},
P_{\clq}M_{\Phi}|_{\clq}) \in \p_n(\clq).
\]
\end{rem}

\section{von-Neumann inequality for finite rank tuples in $\p_n(\clh)$}
\label{section-vn ineq}

Now we turn to the von-Neumann inequality for finite rank tuples in $\p_n(\clh)$.

\begin{thm}\label{vN2}
Let $T\in\p_n(\clh)$ be a finite rank tuple. Then there exists an algebraic variety $V$ in $\overline{\D}^n$ such that
\[
\|p(T)\|\le \sup_{\z\in V}|p(\z)|,
\]
for all $p \in \Comp[z_1,\dots,z_n]$. If, in addition, $T_n$ is a pure contraction, then $V \subseteq \D^n$.
\end{thm}
\begin{proof}
Let $(M_{z_1},\dots, M_{z_{n-1}}, M_{\Phi(z_1,\dots,z_{n-1})})$ on
$H^2_{D_{\h_n}}(\D^{n-1})$ be the isometric dilation of $T$ as in Theorem~\ref{dilation2} and let
\[
\Phi(\z)= A^*+
C^*E(\z)(I_{\clf}-D^*E(\z))^{-1}B^* \quad \quad (\z \in \D^{n-1}),
\]
the transfer function corresponding to the unitary
\[
U^*=\begin{bmatrix} A^*& C^*\\ B^*& D^*\end{bmatrix}: D_{\h_n}\oplus
\clf \to
D_{\h_n}\oplus\clf
\]
as constructed in the proof of Theorem~\ref{dilation2} (see \eqref{eq-U} and note, in view of the assumption that $T$ is finite rank, that $\clk = \{0\}$). Let 
\[
A^*=\begin{bmatrix} W^*& 0\\0&
E^*\end{bmatrix} \; \mbox{on} \; \Hil_0\oplus \Hil_1=D_{\h_n},
\]
be the canonical decomposition of $A^*$ into the
unitary part $W^*$ on $\Hil_0$ and the completely non-unitary part
$E^*$ on $\Hil_1$. With
respect to the above decomposition of $A^*$, let
\[
\Phi(\z)=\begin{bmatrix} \Phi_0(\z)& 0\\0& \Phi_1(\z)\end{bmatrix}
\]
be the decomposition of $\Phi$, where
\[
\Phi_0(\z)\equiv W^* \quad \quad (\z\in\D^{n-1}),
\]
and
\[
\Phi_1(\z)=E^*+C^*E(\z)(I_{\clf}-D^*E(\z))^{-1}B^* \quad \quad (\z\in\D^{n-1}),
\]
is a multiplier in $H^{\infty}_{\mathcal{B}(\Hil_1)}(\D^{n-1})$. We set
\[
V_0:=\{\z\in\bar{\D}^n: \det(z_n I_{\Hil_0}-\Phi_0(z_1,\dots,z_{n-1}))=0\}= \D^{n-1}\times \sigma(W^*),
\]
and
\[
V_1:=\{\z\in\D^n: \det (z_nI_{\Hil_1}-\Phi_1(z_1,\dots,z_{n-1}))=0\}.
\]
Then for each $p \in \Comp[z_1,\dots,z_n]$ we have
\[
 \begin{split}
\|p(T)\| &\le \|p(M_{z_1},\dots,M_{z_{n-1}}, M_{\Phi(z_1,\dots,z_{n-1})})\|
\\
&=\| M_{p(z_1,\dots,z_{n-1},\Phi(z_1,\dots,z_{n-1}))}\|
\\
& = \sup_{\theta_1,\dots,\theta_{n-1}}\| p(e^{i\theta_1}I_{D_{\h_n}},\dots,e^{i\theta_{n-1}}I_{D_{\h_n}},\Phi(e^{i\theta_1},\dots,e^{i\theta_{n-1}}))\|.
\end{split}
\]
Moreover, for $j=0,1$, we have
\[
\begin{split}
&\sup_{\theta_1,\dots,\theta_{n-1}}\| p(e^{i\theta_1}I_{H_j},\dots,e^{i\theta_{n-1}}I_{H_j},\Phi_j(e^{i\theta_1},\dots,e^{i\theta_{n-1}}))\| \\
&= \sup_{\theta_1,\dots,\theta_{n-1}}\{|p(e^{i\theta_1},\dots,e^{i\theta_{n-1}},\lambda)|: \lambda\in \sigma(\Phi_j(e^{i\theta_1},\dots,e^{i\theta_{n-1}}))\}\\
& \le \|p\|_{\partial V_j}.
\end{split}
\]
Since $\Phi(e^{i\theta_1},\dots,e^{i\theta_{n-1}})=
\Phi_0(e^{i\theta_1},\dots,e^{i\theta_{n-1}})\oplus
\Phi_1(e^{i\theta_1},\dots,e^{i\theta_{n-1}})$, we have by
continuity and Lemma~\ref{transfer} that
\[
\|p(T)\|\le \sup_{\z\in V}|p(\z)|,
\]
where
\[
V= V_0\cup V_1.
\]
For the second part, we prove that $\clh_0=\{0\}$ which would imply that $V_0$ is the empty set. Let $\clq=\Pi(\Hil)$, where $\Pi$ is the canonical dilation map for $\h_n$. We claim that
\[
\clq \subseteq H^2_{\clh_1}(\D^{n-1}).
\]
Indeed, let $g\in H^2_{\clh_0}(\D^{n-1})$, $m \in \Z$ and let $g_m =
M_{\Phi_0}^{*m} g$. Then $g_m \in H^2_{\clh_0}(\D^{n-1})$ and
\[
g= M_{\Phi_0}^{ m} g_m.
\]
For $f\in\clq$, we have
\[
\langle g,f\rangle  = \langle M_{\Phi_0}^{m}g_m, f\rangle = \langle g_m, M_{\Phi_0}^{*m} f \rangle = \langle g_m, M_{\Phi}^{*m} f \rangle = \langle g_m , T_n^{* m} f \rangle \raro 0,
\]
as $m\to\infty$, as $T_n$ is pure. This implies that $\clq \subseteq H^2_{\clh_1}(\D^{n-1})$.
By the minimality of the isometric dilation of $\h_n$ (see Theorem \ref{th-dc dil}), we have
\[
\bigvee_{\bm{k} \in \Z^{n-1}} M_{\z}^{\bm{k}}\clq =
H^2_{\cld_{\h_n}}(\D^{n-1}).
\]
Clearly, $H^2_{\clh_1}(\D^{n-1}) \subseteq
H^2_{\cld_{\h_n}}(\D^{n-1})$ is a joint reducing subspace for $(M_{z_1},
\ldots, M_{z_{n-1}})$. Then we have $H^2_{\clh_0}(\D^{n-1})=\{0\}$ and
hence, $\clh_0=\{0\}$. This completes the proof of the theorem.
\end{proof}

The finiteness assumption for tuples in $\p_n(\clh)$ seems natural for refined (in terms of algebraic  varieties) von Neumann inequality. However, we do not know how to prove the existence of (explicit) isometric dilations for tuples without the finiteness assumption.

\vspace{0.3in}

\noindent\textbf{Acknowledgement:} The research of
the second named author is supported by DST-INSPIRE Faculty
Fellowship No. DST/INSPIRE/04/2015/001094. The research of the third named author is supported in part by NBHM grant NBHM/R.P.64/2014, and the Mathematical Research Impact Centric Support (MATRICS) grant, File No: MTR/2017/000522 and Core Research Grant, File No: CRG/2019/000908, by the Science and Engineering Research Board (SERB), Department of Science \& Technology (DST), Government of India.


\begin{thebibliography}{99}


\bibitem{AM1}
J. Agler and J. McCarthy, \emph{Distinguished Varieties}, Acta. Math.
194 (2005), 133-153.

\bibitem{AEM}
C. G. Ambrozie, M. Engli\v{s} ans V. M\"{u}ller, {\em Operator tuples
and analytic models over general domains in $\Comp^n$}, J. Operator
Theory 47 (2002),  287--302.

\bibitem{An}
T. Ando, {\em On a pair of commutative contractions}, Acta Sci.
Math. (Szeged) 24 (1963), 88--90.

\bibitem{BLTT}
J. A. Ball, W.S. Li, D. Timotin and T. T. Trent, {\em A commutant
lifting theorem on the polydisc}, Indiana Univ. Math. J. 48 (1999),
653--675.

\bibitem{BDHS}
S. Barik, B.K. Das, K.J. Haria and J. Sarkar,
{\em Isometric dilations and von Neumann inequality for a class of tuples in the polydisc}, Trans. Amer. Math. Soc. 372 (2019), 1429--1450.

%\bibitem{BCL}
%C. A. Berger, L. A. Coburn and A. Lebow, {\em Representation and
%index theory for $C^*$-algebras generated by commuting isometries},
%J. Funct. Anal. 27 (1978), no. 1, 51--99.

\bibitem{CD}
M.-D. Choi and K. R. Davidson, {\em A $3 \times 3$ dilation
counterexample}, Bull. London Math. Soc. 45 (2013), 511--519.



\bibitem{CD1}
M. Crabb and A. Davie, {\em Von Neumann's inequality for Hilbert
space operators}, Bull. London Math. Soc. 7 (1975), 49--50.

\bibitem{CV1}
R. E. Curto and F.-H. Vasilescu, {\em Standard operator models in
the polydisc}, Indiana Univ. Math. J. 42 (1993), 791-810.

\bibitem{CV2}
R. E. Curto and F.-H. Vasilescu, {\em Standard operator models in the
polydisc}, II, Indiana Univ. Math. J. 44 (1995), 727-746.


\bibitem{DS}
B. K. Das and J. Sarkar, {\em Ando dilations, von Neumann inequality,
and distinguished varieties}, J. Funct. Anal. 272 (2017), 2114-2131.

\bibitem{DSS}
B. K. Das, J. Sarkar and S. Sarkar, {\em Factorizations of
contractions}, Adv. Math 322 (2017), 186-200.

\bibitem{Dr}
S. Drury, {\em Remarks on von Neumann's inequality}, Banach spaces,
harmonic analysis, and probability theory, Storrs, CT, 1980/1981,
Lecture Notes in Mathematics 995 (Springer, Berlin, 1983) 14--32.

\bibitem{FF}
C. Foias and A. Frazho, {\em The commutant lifting approach to interpolation problems}, Operator Theory: Advances and Applications, 44. Birkhuser Verlag, Basel, 1990.

\bibitem{VV}
A. Grinshpan, D.S. Kaliuzhnyi-Verbovetskyi, V. Vinnikov and H.J.
Woerdeman, {\em Classes of tuples of commuting contractions
satisfying the multivariable von Neumann inequality}, J. Funct.
Anal. 256 (2009), 3035-3054.

\bibitem{Ho1}
J. A. Holbrook, {\em Inequalities of von Neumann type for small
matrices}, Function spaces, Edwardsville, IL, 1990, Lecture Notes in
Pure and Applied Mathematics 136 (Dekker, New York, 1992), 189--193.


\bibitem{Ho2}
J. A. Holbrook, {\em Schur norms and the multivariate von Neumann
inequality}, Recent advances in operator theory and related topics,
Szeged, 1999, Operator Theory: Advances and Applications 127
(Birkh\"{a}user, Basel, 2001), 375--386.

\bibitem{K1}
G. Knese, {\em The von Neumann inequality for $3 \times 3$
matrices}, Bull. Lond. Math. Soc. 48 (2016), 53--57.

\bibitem{LK}
{\L}. Kosi\'{n}ski, {\em Three-point Nevanlinna Pick problem in the
polydisc}, Proc. Lond. Math. Soc. 111 (2015), 887--910.


\bibitem{MV}
V. M\"{u}ller and F.-H. Vasilescu, {\em Standard models for some
commuting multioperators}, Proc. Amer. Math. Soc. 117 (1993),
979-989.

\bibitem{Sarason}
D. Sarason, {\em Generalized interpolation in $H^\infty$}, Trans. Amer. Math. Soc. 127 (1967), 179-–203.

\bibitem{NF}
B. Sz.-Nagy and C. Foias, {\em Harmonic Analysis of Operators on
Hilbert Space. North-Holland}, Amsterdam-London, 1970.


\bibitem{Par}
S. Parrott, { \em Unitary dilations for commuting contractions},
Pacific J. Math. 34 (1970), 481--490.


\bibitem{V}
N. Th. Varopoulos, {\em On an inequality of von Neumann and an
application of the metric theory of tensor products to operator
theory}, J. Funct. Anal. 16 (1974), 83--100.


\end{thebibliography}
\end{document}